\documentclass[12pt,reqno]{amsart}
\usepackage{a4,amssymb,amsthm,amscd,amsmath,verbatim,url,enumerate,mathdots,cleveref}
\usepackage[pdftex,usenames,dvipsnames]{color}
\usepackage{mathabx}
\usepackage{fancyhdr}
\usepackage{tikz}
\usepackage{tikz-cd}

\title{A ruled residue theorem for function fields of conics}
\author{Parul Gupta}
\author{Karim Johannes Becher}
\email{parul.gupta@iiserpune.ac.in, parul.gupta@uantwerpen.be}
\email{karimjohannes.becher@uantwerpen.be}
\address{IISER Pune, Dr.~Homi Bhabha Road, Pashan, Pune 411 008, India}
\address{Universiteit Antwerpen, Departement Wiskunde, Middelheim\-laan~1, 2020 Antwerpen, Belgium.}

\thanks{This work was supported by the FWO Odysseus Programme (project \emph{Explicit Methods in Quadratic Form Theory}), funded by the Fonds Wetenschappelijk Onderzoek -- Vlaanderen, and by the {Bijzonder Onderzoeksfonds} (BOF), {Universiteit Antwerpen},
 (project BOF-DOCPRO4, 2865, \emph{Nieuwe methoden in de arithmetiek van lichamen en de theorie van kwadratische vormen}).}
\date{02.09.2020}

\newcommand{\nat}{\mathbb{N}}

\newcommand{\mc}[1]{\mathcal{#1}}

\newcommand{\mg}[1]{{#1}^{\times}}
\newcommand{\sq}[1]{{#1}^{\times 2}}

\newcommand{\ovl}{\overline}

\newcommand{\ind}{\mathsf{ind}}

\renewcommand{\deg}{\mathsf{deg}}

\renewcommand{\max}{\mathsf{max}}
\renewcommand{\min}{\mathsf{min}}

\renewcommand{\setminus}{\smallsetminus}

\swapnumbers
\numberwithin{equation}{section}
\newtheorem*{thm*}{Theorem}
\newtheorem*{thmI}{Theorem I}
\newtheorem*{thmII}{Theorem II}
\newtheorem{thm}[equation]{Theorem}
\newtheorem{prop}[equation]{Proposition}
\newtheorem{cor}[equation]{Corollary}

\newtheorem{lem}[equation]{Lemma}

\theoremstyle{definition}

\renewenvironment{proof}{\par\noindent {\em Proof:}}{\hfill$\Box$\medskip}
\theoremstyle{plain}

\begin{document}
\maketitle

\begin{abstract}
The ruled residue theorem characterises residue field extensions for valuations on a rational function field. 
Under the assumption that the characteristic of the residue field is different from $2$ this theorem is extended here to function fields of conics. The main result is that there is at most one extension of a valuation from on the base field to the function field of a conic for which the residue field extension is transcendental but not ruled. Furthermore the situation when this valuation is present is characterised.
\smallskip

\noindent
{\sc{Classification (MSC 2010):}} 12F20, 12J10, 12J20, 14H05, 16H05
\medskip

\end{abstract}

\section{Introduction}\label{introduction}

By an \emph{algebraic function field} we mean a finitely generated field extension of transcendence degree one. 
We assume familiarity with valuation theory over fields, as covered by \cite[Chapters 2-3]{EP}.
Given a valuation $v$ on a field $E$ with residue field $\kappa$ and an algebraic function field $F/E$ we are interested in the algebraic function fields over $\kappa$ that can occur as residue fields of valuations on $F$ extending $v$.

Before we can make our question more precise  we need to fix some terminology.
An algebraic function field $F/E$ is called \emph{rational} if $F=E(x)$ for some $x\in F$ (necessarily transcendental over $E$).
An algebraic function field $F/E$ is called \emph{ruled} if $F$ is a rational function field over some finite extension of $E$, and it is called \emph{regular} if $E$ is relatively algebraically closed in $F$.
In particular an algebraic function field is rational if and only if it is ruled and regular.

Let $E$ denote a field, $v$ a valuation on $E$ and $\kappa$ the residue field of $v$.
For $x\in E$ with $v(x)\geq 0$ we denote by $\ovl{x}$ the residue of $x$ in $\kappa$ (i.e.~the reduction modulo the maximal ideal of the valuation ring of $v$).
Let $F/E$ be an algebraic function field.
An extension of $v$ to a valuation on $F$ is called \emph{residually transcendental} if its residue field  is a transcendental extension of $\kappa$;
in this case it follows (e.g.~by using \cite[Corollary 2.2.2 and Corollary 3.2.3]{EP}) that this residue field is an algebraic function field over $\kappa$.
Let $w$ be a residually transcendental extension of $v$ to $F$ and let $\kappa'$ be the residue field of $w$.
The Ruled Residue Theorem due to J.~Ohm \cite{Ohm} asserts that, if the extension $F/E$ is ruled, then so is the residue field extension $\kappa'/\kappa$.
The aim of this article is to extend Ohm's result as far as possible to the more general case where $F/E$ has genus zero. 

We assume in the sequel that the residue field $\kappa$ has characteristic different from $2$. 
We further assume that $F/E$ is regular, which can always be achieved by replacing $E$ by the full constant field of $F/E$. 
Assume that $F/E$ has genus zero, or equivalently that $F$ is the function field of a smooth projective conic over $E$ (see \Cref{P:fufi-gen0-conic}). 
Using Ohm's theorem it is easy to conclude that the residue field $\kappa'$ must be the function field of a smooth conic over some finite extension of $\kappa$.
In \cite{KG} Khanduja and Garg obtained a more precise version of this statement.
In \cite[Theorems 1.1 and 1.2]{KG} it is proven
 that, if the residue field extension $\kappa'/\kappa$ is not ruled, then $w$ is an unramified extension of~$v$ and $\kappa'$ is the function field of a smooth projective conic over $\kappa$ without rational point. 
Our main result here is that this can happen for at most one extension of $v$ to $F$.

\begin{thmI}
Let $F/E$ be a function field of genus zero and let $v$ be a valuation on $E$ with residue field $\kappa$ of characteristic different from $2$. 
Then $v$ has at most one extension to $F$ whose residue field is transcendental and not ruled over $\kappa$.
Moreover, if such an extension exists, then it is unramified and  its residue field is a regular algebraic function field of genus zero over $\kappa$.
\end{thmI}

The last part of this statement corresponds essentially to \cite[Theorem 1.2]{KG}.
The proof which we present in Section 3 is independent, but equally elementary and directly based on the methods from \cite{Ohm}, which we revisit in Section 2.
Theorem~I appears as \Cref{C:main1} below in the exposition.

In Section 4 we characterise the situation when there exists an  extension of $v$ to $F$ with a  non-ruled transcendental residue field extension and we relate this to the properties of a quaternion algebra over $E$ which is associated with $F/E$.
To an $E$-quaternion algebra $Q$ we denote by $E(Q)$ the function field of the Severi-Brauer variety given by $Q$.
This is a regular algebraic function field of genus zero, and conversely every regular function field of genus zero over $E$ is isomorphic to $E(Q)/E$ for some $E$-quaternion algebra $Q$.
In this setup we obtain a refined version of our result, which appears as \Cref{T:main} below.

\begin{thmII}\label{maintheorem}
Let $Q$ be an $E$-quaternion algebra and $F=E(Q)$.
Let  $v$ be a valuation on $E$ with residue field $\kappa$ of characteristic different from $2$.
Then the following are equivalent:
\begin{enumerate}[$(i)$]
\item The valuation $v$ extends to a valuation on $F$ whose residue field is transcendental and not ruled over $\kappa$.
\item The valuation $v$ extends uniquely to a valuation on $F$ whose residue field is transcendental and regular over $\kappa$.
\item The valuation $v$ has an unramified extension to a valuation on $Q$.
\end{enumerate}
If these conditions are satisfied, then the valuation in $(i)$ is unique and it coincides with the valuation characterised in $(ii)$.
\end{thmII}

As of now our results are limited to the case where the characteristic of the residue field $\kappa$ (and hence also of the fields $E$ and $F$) is different from $2$.
We expect that a similar result can be proven without this restriction, and we have chosen to formulate our results in a way that we hope them to extend to the general case. 
In particular, this has motivated us to formulate the criterion for the presence of an extension of the valuation $v$ to $F$ in Theorem II with the distinguishing properties $(i)$ and $(ii)$ in terms of the $E$-quaternion algebra $Q$ associated with the function field $F/E$. In characteristic different from $2$, one could formulate the same statement in terms of residue forms of the ternary quadratic form given by the pure part of the norm form of $Q$, which defines the conic over $E$ with function field $F/E$.
Moreover the formulation of Theorem~II in terms of a quaternion algebra raises a question, brought to our attention by Jean-Pierre Tignol and independently by an anonymous referee. Is it possible without assuming conditions $(i)$--$(iii)$ to hold to associate naturally to $v$ a distinguished value function on $Q$ which gives the valuation extension of $v$ to $Q$ in $(iii)$ provided that $(i)$ and $(ii)$ hold?

\section{Valuations on rational function fields}\label{valonRFF}

Let $E$ be a field.
We denote by $E[X]$ the polynomial ring and by $E(X)$ the rational function field in one variable $X$ over $E$.

Note that any element $Y \in E(X) \setminus \{0\}$ has a unique representation $Y =\frac{f}{g}$ with coprime polynomials $f,g \in E[X]$ such that $f$ is monic.
We recall the follwing basic statement on the degree of the extension $E(X)/E(Y)$ for $Y\in E(X)\setminus E$ and the structure of $E[X,Y]$ as an $E[Y]$-module.
In view of our lack of a reference for the second part, we include a full argument.

\begin{prop}\label{polyrep}
Let $Y\in E(X)\setminus E$ and let  $f,g \in E[X]$ be coprime and such that $Y =\frac{f}{g}$.
Then 
$$[E(X):E(Y)]= \max\{\deg(f), \deg(g)\}\,.$$ 
Moreover, $X$ is integral over $E[Y]$ if and only if $\deg(f)>\deg(g)$, and in this case  an $E[Y]$-basis of
$E[X,Y]$ is given by $(1,X,\dots,X^{n-1})$ for $n=[E(X):E(Y)]$.
\end{prop}

\begin{proof}
Since the representation $Y=\frac{f}g$ determines $f$ and $g$ up to a scalar,
we may assume that $f$ is monic. Set $n=\max\{\deg(f), \deg(g)\}$.
Note that $n\geq 1$ because $Y\notin E$.
We consider the polynomial $$H(T) = f(T) -Yg(T)\in E[Y,T]\,.$$
Then $\deg_T(H(T))=n$ and $H(X) =0$.
We claim that $H(T)$ is irreducible in  $E(Y)[T]$.
Since $Y\notin E$, $H(T)$ is primitive as a polynomial in $T$ over $E[Y]$.
By Gauss' Lemma, it is enough to show that $H(T)$ is irreducible in $E[Y,T]$.
Suppose that $H(T) = h_1h_2$ for some $h_1,h_2 \in E[Y,T]$. 
Since $\deg_Y(H(T)) = 1$ we have $\deg_Y(h_1)=0$ or $\deg_Y(h_2)=0$.
Assume that $\deg_Y(h_1)=0$.
Then $h_1 \in E[T]$ and $h_1$ divides $f$ and $g$ in $E[T]$.
As $f$ and $g$ are coprime, we obtain that $h_1\in E$.
This proves that $H(T)$ is irreducible.
Hence $[E(X):E(Y)] = \deg_T(H(T))=n$.

If $\deg(f)>\deg(g)$, then $H(T)$ is monic as a polynomial in $T$ over $E[Y]$, and as $H(X) =0$, it follows that $X$ is integral over $E[Y]$ and $(1,X,\dots,X^{n-1})$ is an $E[Y]$-basis of $E[X,Y]$ for $n=\deg_T(H(T))=[E(X):E(Y)]$.

Assume conversely that $X$ is integral over $E(Y)$.
Let $\theta \in E[Y]$ be the leading coefficient of $H(T)$. Then $\theta^{-1}H(T)$ is the minimal polynomial  of $X$ over $E(Y)$. Since $X$ is integral over $E[Y]$, we have that $\theta^{-1}H(T) \in E[Y,T]$. 
Since $H(T)$ is primitive in $T$ over $E[Y]$, we obtain that $\theta\in\mg{E}$. 
Hence $\deg(f) > \deg(g)$.
\end{proof}

For a valuation $v$ on a field $E$, we denote by $\mc{O}_v$ the valuation ring, by $\kappa_v$ the residue field and by $\Gamma_v$ the value group of $v$. For $x\in {\mc O_v}$ we denote by $\ovl{x}$ the residue of $x$ in $\kappa_v$.

We will now collect some facts on extensions of valuations from $E$ to a field extension. In the final section we will also consider extensions of $v$ to quaternion division algebras over $E$. 

Let $v$ be a valuation on $E$.
Let $L/E$ be a field extension or, more generally, let $L$ be an $E$-division algebra. 
An \emph{extension of $v$ to $L$} is a valuation $w$ on $L$ such that $\Gamma_v \subseteq \Gamma_w$ and $w|_E=v$.
By \cite[Section 3.1]{EP}, if $L$ is a field, then such an extension always exists.
Given an extension $w$ of $v$ to $L$, we denote 
the value group by $\Gamma_w$ and the residue field by $\kappa_w$,
and we obtain that $\Gamma_v$ is a subgroup of $\Gamma_w$ and $\kappa_w/\kappa_v$ is a field extension.
We call an extension $w$ of $v$ to $L$ \emph{unramified} if $\Gamma_w=\Gamma_v$ holds, otherwise we call it \emph{ramified}.
We identify two extensions of a valuation to $L$ if they correspond to one another under an order preserving isomorphism of their value groups.

We will often use the fact that, if $L/E$ is a finite field extension and $w$ is an extension of $v$ to $E$, then $[\Gamma_w:\Gamma_v]\cdot [\kappa_w:\kappa_v]\leq [L:E]$; see \cite[Corollary 3.2.3]{EP}.
We will especially use this in the case of a quadratic field extension.
In the case where $v$ is unramified in this quadratic extension we will use the following explicit description of the residue field extension.
 
\begin{lem}\label{L:FI2}
Assume that $v(2)=0$.
Let $a\in \mg E\setminus \sq E$ be such that $v(a)\in 2\Gamma_v$ and let $v'$ be an extension of $v$ to $E(\sqrt{a})$.
Then $\Gamma_{v'}=\Gamma_v$, $a\sq E \cap \mg{\mc O}_v \neq \emptyset$ and $\kappa_{v'}=\kappa_v(\sqrt{\ovl{u}})$ for any $u \in a\sq E \cap \mg{\mc O}_v$.
Moreover, if $v'$ is the unique extension of $v$ to $E(\sqrt{a})$, then $[\kappa_{v'}:\kappa_v]=2$.
\end{lem}
\begin{proof}
Let $L = E(\sqrt{a})$. Since $[L:E]=2$, it follows by \cite[Corollary 3.2.3]{EP} that 
$[\kappa_{v'}:\kappa_v]\cdot [\Gamma_{v'}:\Gamma_v]\leq 2$. 
As $v(a)\in 2\Gamma_v$ there exists $x\in \mg{E}$ with $v(a)=2v(x)$, and 
then $v(ax^{-2})=0$, whereby $a\sq E \cap \mg{\mc O}_v\neq\emptyset$.

Fix now $u\in a\sq E \cap \mg{\mc O}_v$.
Then $u\in \sq{L}\cap \mg{\mc{O}}_{v'}\,$, whereby $\kappa_v (\sqrt{\ovl{u}}) \subseteq \kappa_{v'}$.
If $\ovl{u} \notin \sq{\kappa}_v$, then $[\kappa_v(\sqrt{\ovl{u}}): \kappa_v] =2$, and as $[\kappa_{v'}:\kappa_v]\leq 2$
we conclude that $\kappa_{v'} =\kappa_v(\sqrt{\ovl{u}})$, $[\kappa_{v'}:\kappa_v]=2$ and $\Gamma_{v'} = \Gamma_v$.
Assume now that $\ovl{u} \in \sq{\kappa}_v$.
Let $y \in \mg {\mc O}_v$ be such that $\ovl{u} = \ovl{y}^2$.
Then $$0 < v'(y^2-u) = v'((y+\sqrt{u})(y-\sqrt{u})) = v'(y+\sqrt{u}) + v'(y-\sqrt{u})$$
and
$$0 = v'(2y) \geq \min \{ v'(y+\sqrt{u}), v'(y-\sqrt{u})\}\,.$$
This implies that $v'(y+\sqrt{u}) =0$ or $v'(y-\sqrt{u}) =0$.
From the last two conclusions we obtain that $v'(y+\sqrt{u}) \neq v'(y-\sqrt{u})$.
Let $\sigma$ be the nontrivial $E$-automorphism of $L$. Then $v'\circ \sigma$ is a 
valuation on $L$ extending $v$ and different from $v'$. 
By \cite[Theorem 3.3.4]{EP} this implies that $\kappa_{v'} = \kappa_v = \kappa_v(\sqrt{\ovl{u}})$.  
\end{proof}

For the sequel we fix a field $E$ and a valuation $v$ on $E$.
We denote by $\kappa$ the residue field of $v$ and by $\Gamma$ the value group of $v$.
The following statement gives an extension of $v$ from $E$ to the rational function field $E(X)$.

\begin{prop}\label{gaussextdef}
There exists a unique valuation $w$ on $E(X)$ with $w|_E=v$, $w(X) = 0$ and such that the residue $\ovl{X}$  of $X$ in $\kappa_w$ is transcendental over $\kappa$. 
For this valuation $w$ we have that $\kappa_w =\kappa(\ovl{X})$ and $w(\mg {E(X)}) = v(\mg E)$.
For $n\in\nat$ and  $a_0,\dots,a_n\in E$ we have 
$w(\mbox{$\sum_{i=0}^n$}a_iX^i)= \min\{v(a_0),\dots,v(a_n)\}$.  
\end{prop}

\begin{proof} 
See  \cite[Corollary 2.2.2]{EP}.
\end{proof}

The valuation on $E(X)$ defined in \Cref{gaussextdef} is called the \emph{Gauss extension of $v$ to $E(X)$ with respect to $X$}\index{Gauss extension}. 
More generally, we call a valuation on $E(X)$ a \emph{Gauss extension of $v$} if it is equal to the Gauss extension with respect to $Y$ for some  $Y\in E(X)$ with $E(Y)=E(X)$.

\begin{lem}\label{improvetransele}
Let $w$ be an extension of $v$ to $E(X)$.
Let $Y \in E(X) \cap \mg{\mc O}_w$ be such that $\ovl{Y}$ is transcendental over $\kappa$.
Then there exists $Y' \in E(X) \cap \mg{\mc O}_w$ such that  $E(Y) = E(Y')$, $\ovl{Y'}$ is transcendental over $\kappa$, $\kappa (\ovl{Y}) =\kappa(\ovl{Y'}) $ and $X$ is integral over~$E[Y']$. 
 \end{lem}
 
\begin{proof}
The statement follows from \cite[Lemma 3.1]{Ohm} and \Cref{polyrep}.
\end{proof}

Let $F/E$ be an algebraic function field over $E$.
An extension $w$ of $v$ to $F$ is called \emph{residually transcendental} if the residue field extension $\kappa_w/\kappa$ is transcendental. 
Let $w$ be a residually transcendental extension of $v$ to $F$.
Then $\kappa_w/\kappa$ is an algebraic function field.
We define 
$$\ind(w/E)= \min\{[F: E(Y)]\,\mid\, Y \in \mg{\mc O}_w\mbox{ and~} \ovl{Y} \mbox{ is transcendental over } \kappa \},$$
and call this the \emph{Ohm index of $w$ over $E$}.
We observe that $\ind(w/E)=1$ if and only if $F$ is a rational function field and $w$ is a Gauss extension to $F$.

We call $Y\in \mg{\mc{O}}_w$ an \emph{Ohm element of $w$ over $E$} if
$\ovl{Y}$ is transcendental over $\kappa$
and $[F:E(Y)] =\ind(w/E)$.
Note that for any residually transcendental extension of a valuation to an algebraic function field there exists an Ohm element.
Our definition is motivated by Ohm's method of proof for the Ruled Residue Theorem, where these elements 
play a prominent role. 
We will the following refined formulation of his result.

\begin{thm}[Ohm]\label{RRT}
Let $w$ be a residually transcendental extension of $v$ to $E(X)$.
Let $\ell$ be the relative algebraic closure of $\kappa$ in $\kappa_w$. 
Then $\kappa_w$ is a rational function field over $\ell$.
More precisely,  $\kappa_w = \ell( \ovl{Y})$ for any Ohm element $Y$ of $w$ over $E$.
\end{thm}

\begin{proof}
See \cite[Theorem 3.3]{Ohm} and its proof.
\end{proof}  

For $n \in \nat$, we set
$$E[X]_n  = \{ g \in E[X] \mid  \deg(g ) \leq n\}\,.$$

\begin{lem}\label{algebraicele}
Let $w$ be a residually transcendental extension of $v$ to $E(X)$ and let $n=\ind(w/E)$.
Let $Z\in \mc{O}_w$ be such that $Z=\frac{f}{g}$ for some $f,g\in E[X]_{n-1}\setminus\{0\}$.
Then the residue $\ovl{Z}$ is algebraic over $\kappa$.
\end{lem}

\begin{proof}
If $Z\in E$ or $w(Z)>0$, then $\ovl{Z} \in \kappa$.
Assume now that $Z\in E(X)\setminus E$ and $w(Z)=0$.
Then, by \Cref{polyrep}, $[E(X):E(Z)]\leq \max\{\deg(f),\deg(g)\}<n$. 
Since $n=\ind(w/E)$, we obtain that $\ovl Z$ is algebraic over $\kappa$.
\end{proof}

The following lemma is distilled from the proof of \cite[Theorem 3.3]{Ohm}.
The last part of the statement of \Cref{P:Ohm-val-smallpol} (concerning the value group) is also obtained in \cite[Lemma 2.2]{KG}.

\begin{lem}\label{P:Ohm-val-smallpol}
Let $w$ be a residually transcendental extension of $v$ to $E(X)$.
Let $n=\ind(w/E)$.
Then for any $m\in\nat$ and $g_0,\dots,g_m\in E[X]_{{n}-1}$ and for any Ohm element $Y$ of $w$ over $E$, we have that
$$w(g_0+g_1Y+\dots+g_mY^m)=\min\{w(g_i)\mid 0\leq i\leq m\}\,.$$
Furthermore $$\Gamma_w=w(E[X]_{n-1}\setminus\{0\})\,.$$
\end{lem}

\begin{proof}
Let $m\in\nat$ and $g_0,\dots,g_m\in E[X]_{n-1}$ and
let $Y$ be an Ohm element of $w$ over $E$.
Let $\gamma = \min\{ w(g_i)\mid 0\leq i\leq m\}$. Let $j\in\{0,\dots,m\}$ be such that $w(g_i)>\gamma$ for $0\leq i<j$ and 
$w(g_{j})=\gamma$.

For $i \in \{0 , \ldots, m\}$, we have $\frac{g_i}{g_{j}}\in \mc{O}_w$, and
since $g_j,g_i\in E[X]_{n-1}$,
it follows by \Cref{algebraicele} that $\ovl{\frac{g_i}{g_{j}}}$ is algebraic over $\kappa$.

Set $g = g_0+g_1Y+\dots+g_mY^m$ and $\vartheta=\frac{g}{g_j}$.
Note that $\vartheta\in \mc{O}_w$.
Let $\ell$ denote the relative algebraic closure of $\kappa$ in $\kappa_w$.
Then
$\ovl{\vartheta}\equiv \ovl{Y}^{j}\bmod \ovl{Y}^{j+1} \mbox{ in } \ell[\ovl{Y}]$.
Since $Y$ is an Ohm element of $w$ over $E$, we have that $\ovl Y$ is transcendental over $\ell$.
Since $\ell[\ovl{Y}]\subseteq\kappa_w$, 
we obtain in particular  
that $\ovl{\vartheta}\neq 0$ in $\kappa_w$, whereby $w(g)=w(g_{j})=\gamma$.
This shows the first part of the statement.

To show the second part, we note first that it follows from \Cref{improvetransele} that there exists an Ohm element $Y$ of $w$ over $E$ such that $X$ is integral over $E[Y]$.
Since $\ovl{Y}$ is transcendental over $\kappa$, by \Cref{gaussextdef} $w|_{E(Y)}$ is the Gauss extension of $v$ with respect to $Y$,
and in particular $w(\mg{E(Y)})=\Gamma$.
It follows by \cite[Corollary 3.2.3]{EP} that
$$[\Gamma_w:\Gamma]=[\Gamma_w:\Gamma_{w|_{E(Y)}}]\leq [E(X):E(Y)]<\infty\,.$$

Let $\Delta$ denote the image of $E[X]\setminus\{0\}$ under $w$.
Then $\Delta$ is closed under addition, $\Gamma\subseteq\Delta$ and $\Gamma_w$ is generated by $\Delta$.
Since $[\Gamma_w:\Gamma]$ is finite, it follows that $\Delta=\Gamma_w$.
This implies that $\Gamma_w=w(\{g\in E[X]\setminus\{0\}\})$.

Let $\gamma\in \Gamma_w$. 
There exists $g\in E[X]$ such that $w(g) =\gamma$.
By \Cref{polyrep}, there exist $m\in \nat$ and $g_0,\dots,g_m\in E[X]_{n-1}$ such that 
$g  =g_0+ \ldots +g_mY^m$.
Now the first part of the statement yields that $\gamma = w(g) =\min\{w(g_i)\mid 0\leq i\leq m\}$.
This shows that $\Gamma_w=w(E[X]_{n-1}\!\setminus\!\{0\})$.
\end{proof}

For a finite field extension $E'/E$ we denote by $\mathsf{N}_{E'/E}:E'\to E$ the norm map.
In the following statement we compare for a separable quadratic field extension $E'/E$ the Ohm index of a residually transcendental extension of $v$ to $E'(X)$  with the Ohm index of its restriction to $E(X)$.

\begin{lem}\label{SRTQ}
Assume that  $v(2)=0$.
Let $b \in \mg E\setminus \sq E$ and $E' =E(\sqrt{b})$. Let $w'$ a residually transcendental extension of $v$ to $E'(X)$ and $w=w'|_{E(X)}$.
Let $\ell$ be the relative algebraic closure of $\kappa$ in $\kappa_{w'}$.
Assume that $[\kappa_{w'}: \kappa_{w}] =2$ and $\ell\subseteq \kappa_{w}$.
Then $\ind(w/E) > \ind(w'/E')$ and there exists $\varphi\in E(X)\setminus\{0\}$ such that $\sqrt{b}\varphi$ is an Ohm element for $w'$ over $E'$.
\end{lem}

\begin{proof}
Since $[\kappa_{w'}: \kappa_{w}] =2$, it follows by \cite[Theorem~3.4.3]{EP} that  $\Gamma_{w'} = \Gamma_{w}$.
In particular  $w(b) \in 2\Gamma_{w'} = 2\Gamma_{w}$.
We set $\beta=\sqrt{b}\in E'$.
By the hypotheses, every element of $\kappa_{w'}\setminus\kappa_w$ is transcendental over $\kappa$, and by \Cref{L:FI2}, this applies in particular to $\beta\ovl{\varphi}$ for any $\varphi\in E(X)$ with $w'(\beta\varphi)=0$.

By \Cref{P:Ohm-val-smallpol}, there exists $h\in E[X]\setminus\{0\}$ such that $\deg(h) < \ind(w/E)$ and $w(h) = -\frac{1}{2}w(b)=-w'(\beta)$. 
Then $\ovl{\beta h}$ is transcendental over $\kappa$ and therefore
$$\ind(w/E) > \deg(h)= [E'(X): E'(\beta h)] \geq \ind(w'/E').$$
In particular, no element of $E(X)$ can be an Ohm element of $w'$ over $E'$. 

We set $n=\ind({w'}/E')$.
By \Cref{improvetransele}, there exists an Ohm element $Y$ of $w'$ over $E'$.
We write $Y = \frac{f}{g}$ for certain coprime polynomials $f, g \in E'[X]\setminus\{0\}$ with $f$ monic.
We obtain that $\max(\deg(f),\deg(g))=[E'(X):E'(Y)]=\ind(w'/E')=n$.
We write $f =f_1+\beta f_2$ and $g = g_1+\beta g_2$ with $f_1, f_2, g_1, g_2 \in E[X]$.

If $w'(f_1) =  w'(\beta f_2)$, then we have that $\ovl{\beta f_2f_1^{-1}}$ is transcendental over $\kappa$, and since $[E'(X):E'(\beta f_2f_1^{-1})]  \leq n=\ind(w'/E')$, we conclude that $\beta f_2f_1^{-1}$ is an Ohm element of $w'$ over $E'$. 
Similarly, if  $w'(g_1) =  w'(\beta g_2)$, then  $\beta g_2g_1^{-1}$ is an Ohm element of $w'$ over $E'$.

Assume now that  $w'(f_1) \neq w'(\beta f_2)$ and $w' (g_1) \neq w'(\beta g_2)$.
In this case the property of $Y$ to be an Ohm element of $w'$ over $E'$ will not be affected if we replace within each pair $(f_1,f_2)$ and $(g_1,g_2)$ the polynomial with the higher $w'$-value by zero.
We further may replace $Y$ by $Y^{-1}$ if necessary.
After these changes, $Y$ is of the form $\varphi$ or $\beta\varphi$ for some $\varphi\in\mg{E(X)}$.
Since $E(X)$ does not contain any Ohm element of $w'$ over $E'$, we get that $\beta\varphi$ is an Ohm element for $w'$ over $E'$. 
\end{proof}

\section{Valuations on function fields of genus zero}

Let $E$ be a field of characteristic different from $2$.
We recall the description of function fields of genus zero as function fields of conics.

\begin{prop}\label{P:fufi-gen0-conic}
An algebraic function field $F/E$ is regular and of genus zero if and only if $F= E(x)(\sqrt{ax^2+b})$ for certain $x\in F$ and $a, b \in \mg E$.
\end{prop}
\begin{proof}
See e.g.~\cite[Sections 17 and 18]{Deu}.
\end{proof}

We assume in the sequel that $v$ is a \emph{nondyadic} valuation on $E$,
that is $v(2)=0$. 
We again denote by $\kappa$ the residue field of $v$ and by $\Gamma$ the value group of $v$.

\begin{lem}\label{L:hypel-nonruled}
Let $F/E(X)$ be a quadratic field extension and $w$ a residually transcendental extension of $v$ to $F$ such that $\kappa_w/\kappa$ is not ruled.
Let $\ell$ be the relative algebraic closure of $\kappa$ in $\kappa_{w}$ and $w_0=w|_{E(X)}$. 
Then $\Gamma_w=\Gamma_{w_0}$ 
 and $\ell\subseteq\kappa_{w_0}$.
\end{lem}
\begin{proof}
By \cite[Corollary 3.2.3]{EP} we have $[\kappa_w:\kappa_{w_0}]\cdot [\Gamma_w:\Gamma_{w_0}]\leq [F:E(X)]=2$.
By \Cref{RRT}, $\kappa_{w_0}/\kappa$ is ruled.
Hence the hypothesis implies that $\kappa_w\neq \kappa_{w_0}$, and it follows that $\Gamma_w=\Gamma_{w_0}$.
As  $\kappa_w/\kappa$ is not ruled, $\kappa_w/\ell$ is not rational.
Since the extension $\kappa_{w_0}/\kappa$ is ruled, the extension $\kappa_{w_0}/(\ell\cap \kappa_{w_0})$ is rational, whereby $\ell\kappa_{w_0}/\ell$ is rational.
Therefore $\ell\kappa_{w_0}\subsetneq\kappa_w$.
Since $[\kappa_w:\kappa_{w_0}]=2$, it follows that $\ell\subseteq \kappa_{w_0}$.\end{proof}

\begin{lem}\label{uniqueextensionlem}
Let $F=E(X)(\sqrt{aX^2+b})$ with $a,b\in \mg{E}$.
Then the Gauss extension of $v$ to $E(X^2)$ with respect to $\frac{aX^2}{b}$ extends uniquely to a valuation $w$ on~$F$.
Moreover, we have $$\Gamma_w=\Gamma\cup (\mbox{$\frac{1}2$}v(a)+\Gamma)\cup (\mbox{$\frac{1}2$}v(b)+\Gamma)\cup(\mbox{$\frac{1}2$}v(ab)+\Gamma)\,,$$
and if $\Gamma_w\neq \Gamma$, then $\kappa_w/\kappa$ is a rational function field.
\end{lem}

\begin{proof}
Let $w_1$ denote the Gauss extension of $v$ to $E(X^2)$ with respect to $\frac{aX^2}{b}$ and let $w$ be an extension of $w_1$ to $F$.
Set $z = \ovl{\frac{aX^2}{b}}$ in $\kappa_{w_1}$.
By \Cref{gaussextdef}, we have that $\Gamma_{w_1} =\Gamma$, $\kappa_{w_1} =\kappa(z)$, and
$w(aX^2+b) = v(b) = w(aX^2)$.
Since $aX^2+b \in \sq F$, we obtain that $w(a), w(b),  w(aX^2+b) \in 2\Gamma_w$, whereby $\frac{1}2v(a),\frac{1}2v(b)\in \Gamma_w$.
We set
$$\Gamma'=\Gamma\cup (\mbox{$\frac{1}2$}v(a)+\Gamma)\cup (\mbox{$\frac{1}2$}v(b)+\Gamma)\cup(\mbox{$\frac{1}2$}v(ab)+\Gamma)\,.$$

Note that $[F:E(X^2)]=4$.
Since $\Gamma_{w_1}=\Gamma$, it follows by \cite[Corollary 3.2.3]{EP} that 
\begin{equation*}
[\Gamma_w:\Gamma]\cdot [\kappa_w:\kappa_{w_1}]  \leq   4 
\end{equation*}
and that, if equality holds here, then $w$ is the unique extension of $w_1$ to $F$.

Note that $\Gamma'\subseteq \Gamma_w$ and $[\Gamma':\Gamma]$ divides $4$.
Hence, in order to show that $w$ is the unique extension of $w_1$ to $F$ and that $\Gamma_w=\Gamma'$, it suffices now to show that 
\begin{equation*}
[\Gamma':\Gamma]\cdot [\kappa_w:\kappa_{w_1}]  \geq   4 \tag{\mbox{$\star$}}\,.
\end{equation*}

\underline{Case 1:} $[\Gamma':\Gamma]=4$. 
Then $(\star)$ holds trivially and we have $\kappa_w=\kappa_{w_1}=\kappa(z)$.
\medskip

\underline{Case 2:} $[\Gamma':\Gamma]=2$.
Then exactly one of the values $v(a), v(b)$ and $v(ab)$ lies in $2\Gamma$ and $[\kappa_w:\kappa_{w_1}]\leq 2$.
If $v(a) \in 2\Gamma$, then we choose $u \in a\sq E \cap \mg {\mc O}_v$ and obtain that 
$\kappa_{w_1}=\kappa(z)\subsetneq\kappa(\sqrt{\ovl{u}z({z+1})})\subseteq\kappa_w$, whereby $\kappa_w=\kappa(\sqrt{\ovl{u}z({z+1})})$.
If $v(b)\in 2\Gamma$, then we choose $u \in b\sq E \cap \mg {\mc O}_v$ and obtain that 
$\kappa_{w_1}=\kappa(z)\subsetneq\kappa(\sqrt{\ovl{u}({z+1})})\subseteq\kappa_w$, 
whereby $\kappa_w=\kappa(\sqrt{\ovl{u}({z+1})})$.
If $v(ab)\in 2\Gamma$, then we choose $u \in ab\sq E \cap \mg {\mc O}_v$ and obtain that 
$\kappa_{w_1}=\kappa(z)\subsetneq\kappa(\sqrt{\ovl{u}{z}})\subseteq\kappa_w$, 
whereby $\kappa_w=\kappa(\sqrt{\ovl{u}{z}})$.

Hence in each of the three subcases, $\kappa_w$ is a rational function field over $\kappa$ and $[\kappa_w:\kappa_{w_1}]=2$, which establishes $(\star)$.

\underline{Case 3:}
$[\Gamma':\Gamma]=1$. Then $v(a), v(b) \in 2\Gamma$.
We choose $u,\nu\in\mg {\mc O}_v$ with $uab,\nu b\in\sq{E}$ and obtain that
$\kappa_w = \kappa (\sqrt{\ovl{u}z}, \sqrt{\ovl{\nu}(z+1)} )$.
Since $\kappa_{w_1}=\kappa(z)$ where $z$ is transcendental over $\kappa$ and since $\ovl{u},\ovl{\nu}\in\kappa$,
we obtain that $[\kappa_w : \kappa_{w_1}] =4$.
This establishes $(\star)$, and it follows that $\Gamma_w=\Gamma'= \Gamma$.
\end{proof}
 
\begin{prop}\label{P:fufi-conic-gauss-residue}
Let $F=E(X)(\sqrt{aX^2+b})$ with $a,b\in \mg{E}$.
Let $w$ be an extension of $v$ to $F$ such that $w|_{E(X^2)}$ is the Gauss extension of $v$ with respect to~$\frac{aX^2}{b}$.
Then $\kappa_w/\kappa$ is a regular algebraic function field of genus zero. More precisely, we have the following two cases:
\begin{enumerate}[$(i)$]
\item If $v(a)\notin 2\Gamma$ or $v(b)\notin 2\Gamma$, then $\kappa_w/\kappa$ is a rational function field.
\item If $v(a),v(b)\in 2\Gamma$, then there exist $a_0\in a\sq{E}$ and $b_0\in b\sq{E}$ such that $v(a_0)=v(b_0)=0$, and for any choice of such $a_0$ and $b_0$ we have that $\kappa_w=\kappa(T)\left(\sqrt{\ovl{a}_0T^2+\ovl{b}_0}\right)$
for some $T\in\kappa_w$ which is transcendental over $\kappa$.
\end{enumerate}
\end{prop}
\begin{proof}
$(i)$\, If $v(a)\notin 2\Gamma$ or $v(b)\notin 2\Gamma$, then
$\Gamma_w\neq \Gamma$, and it follows by \Cref{uniqueextensionlem} that $\kappa_w/\kappa$ is a rational function field.

$(ii)$\,
Assume that $v(a),v(b)\in 2\Gamma$.
Hence there exist $u,\nu\in \mg{E}$ such that $2v(u)=-v(a)$ and $2v(\nu)=-v(b)$. 
We set $a_0=au^2$, $b_0=b\nu^2$ and $X_0=\frac{\nu}{u}X$.
Then $a_0\in a\sq{E}$, $b_0\in b\sq{E}$ and $v(a_0)=v(b_0)=0$.
Moreover $E(X)=E(X_0)$ and 
$F=E(X_0)(\sqrt{a_0X_0^2+b_0})$. 
Let $w_0=w|_{E(X)}$. 
Note that $w(X_0^2)=w(\frac{aX^2}{b})=0$ and thus $w(X_0)=0$.
We conclude that $w_0$ is the Gauss extension of $v$ with respect to $X_0$.
Hence $T=\ovl{X_0}$ in $\kappa_{w_0}$ is transcendental over $\kappa$, and we obtain that $\kappa_{w_0}=\kappa(T)$  and 
$\kappa_w=\kappa_{w_0}\left(\sqrt{\ovl{a}_0T^2+\ovl{b}_0}\right)$.
In particular, $\kappa_w/\kappa$ is a regular algebraic function field of genus zero, by \Cref{P:fufi-gen0-conic}.
\end{proof}

\begin{thm}\label{non-ruled-val}
Let $F=E(X)(\sqrt{aX^2+b})$ with $a,b\in \mg{E}$.
Let $w$ be an extension of $v$ to $F$ such that $\kappa_w/\kappa$ is transcendental and not ruled.
Then $w$ is the unique extension  to $F$ of the Gauss extension of $v$  on $E(X^2)$ with respect to~$\frac{aX^2}b$.
Furthermore $\Gamma_w=\Gamma$. 
\end{thm}
\begin{proof}
If $w(b)>w(aX^2)$, then 
$w(aX^2)=w(aX^2+b)$ and for $u=1+\frac{b}{aX^2}$ we obtain that 
$u\in a\sq{F}\cap\mg{\mc{O}_w\!\!}$ and $\ovl{u}=1$.

If $w(b)<w(aX^2)$,
then $w(b)=w(aX^2+b)$ and for $u=1+\frac{aX^2}{b}$ we obtain that $u\in b\sq{F}\cap\mg{\mc{O}_w\!\!}$ and $\ovl{u}=1$.

If $w(b)=w(aX^2)<w(aX^2+b)$, then
for $u=-\frac{aX^2}b=1-\frac{aX^2+b}b$ we obtain that $u\in -ab\sq{F}\cap\mg{\mc{O}_w\!\!}$ and $\ovl{u}=1$.

We will now show that none of these three inequalities is possible.
Let $c$ be any of the elements $a,b$ and $-ab$.
Then  $F(\sqrt{c})$ is a rational function field over $E(\sqrt{c})$, by~\cite[Proposition 45.1]{EKM}. Hence $F(\sqrt{c})/E$ is ruled.
Consider an extension $w'$ of $w$ to $F(\sqrt{c})$. 
It follows by \Cref{RRT} that the extension $\kappa_{w'}/\kappa$ is ruled.
Since $\kappa_w/\kappa$ is not ruled, we obtain that $\kappa_{w}\subsetneq\kappa_{w'}$.
It follows by \Cref{L:FI2} that, for any $u\in c\sq{F}\cap\mg{\mc{O}_w\!\!}$, we have $\ovl{u}\notin\sq{\kappa_w\!\!}$, so in particular $\ovl{u}\neq 1$.
This shows that none of the three inequalities above can hold.
Hence, so far we have shown that $$w(b)=w(aX^2)=w(aX^2+b)\in 2\Gamma_{w}\,.$$


Let $w_0=w|_{E(X)}$.
By \Cref{L:hypel-nonruled} we have $\Gamma_{w}=\Gamma_{w_0}$. Hence $w(aX^2+b)\in 2\Gamma_{w_0}$.
Let $\ell$ denote the relative algebraic closure of $\kappa$ in $\kappa_{w}$.
By \Cref{L:hypel-nonruled} we have that $\ell\subseteq\kappa_{w_0}$.
For any $u\in(aX^2+b)\sq{E(X)}\cap \mg{\mc{O}}_{w}$, 
we have $\kappa_w=\kappa_{w_0}(\sqrt{\ovl{u}})\neq \kappa_{w_0}$ by \Cref{L:FI2}, and since $\ell\subseteq \kappa_{w_0}$ this implies that $\sqrt{\ovl{u}}\notin \ell$, whereby $\ovl{u}\notin \ell$.

We set $w_{\star}=w|_{E(X^2)}$ and $Z=\frac{aX^2}{b}$. 
Then $Z,Z+1\in\mg{\mc{O}}_w$ and $E(Z)=E(X^2)$. 
We claim that $w_{\star}$ is the Gauss extension of $v$ with respect to~$Z$. 
By \Cref{uniqueextensionlem}, this implies that $w$ is the unique extension of $w_\star$ to $F$ and $\Gamma_w=\Gamma$.
To show the claim we need to show that $\ovl{Z}\notin\ell$.

Suppose on the contrary that $\ovl{Z}\in\ell$.
Then $\ovl{Z},\ovl{Z}+1\in\mg{\ell}$.
We fix an element $u\in(aX^2+b)\sq{E(X)}\cap \mg{\mc{O}}_{w}$ and set $y=u(Z+1)$.
Let $w '$ denote an extension of $w_0$ to $E(\sqrt{b})(X)$. 
As $y\in b\sq{E(X)}\cap\mg{\mc{O}}_{w'}$, it follows 
by \Cref{L:FI2} that $\kappa_{w'} = \kappa_{w_0}(\sqrt{\ovl{y}})$.
Since $\ovl{Z}+1\in\mg{\ell}$ and $\ovl{u}\notin \ell$ we conclude that $\ovl{y}\notin \mg{\ell}$.
Therefore $\ell$ is relatively algebraically closed in $\kappa_{w'}$ and $[\kappa_{w'}: \kappa_{w_0}] =2$.
Hence \Cref{SRTQ} applies and yields an an Ohm element of $w'$ over $E(\sqrt{b})$  of the form $T =\sqrt{b} \varphi$ for some $\varphi \in E(X)\setminus \{0\}$.
It follows by \Cref{RRT} that $\kappa_{w'} = \ell(\ovl{T})$.
Since
$[\ell(\ovl{T}) : \ell(\ovl{T}^2)] =2= [\kappa_{w'}: \kappa_{w_0}]$
and $\ovl{T}^2\in\kappa_{w_0}$, it follows that $\kappa_{w_0}  = \ell(\ovl{T}^2)$.
As $(Z+1)T^2=(aX^2+b)\varphi^2\in (aX^2+b)\sq{E(X)}\cap \mg{\mc{O}}_w$ and $\ovl{Z}\in\ell$, we obtain using \Cref{L:FI2} that 
$$\kappa_w=\kappa_{w_0}\left(\sqrt{({\ovl{Z}+1}){\ovl{T}^{2}}}\right)=\ell(\ovl{T}^2)\left(\sqrt{({\ovl{Z}+1})\ovl{T}^{2}}\right) =\ell\left(\sqrt{({\ovl{Z}+1})\ovl{T}^{2}}\right).$$
In particular $\kappa_w/\kappa$ is ruled, which contradicts the hypothesis.
Thus $\ovl{Z}\notin \ell$.
\end{proof}

\begin{cor}\label{C:main1}
Let $F/E$ be a function field of genus zero and let $v$ be a valuation on $E$ with residue field $\kappa$ of characteristic different from $2$. 
Then $v$ has at most one extension to $F$ whose residue field is transcendental and not ruled over $\kappa$.
Moreover, if such an extension exists, then it is unramified and its residue field is a regular algebraic function field of genus zero over $\kappa$.
\end{cor}
\begin{proof}
In view of \Cref{P:fufi-gen0-conic}, this follows immediately from 
\Cref{non-ruled-val} and \Cref{P:fufi-conic-gauss-residue}.
\end{proof}

\section{Relation to quaternion algebras} 

Let $E$ be a field of characteristic different from $2$. 
To elements $a,b\in\mg{E}$ one associates the $E$-quaternion algebra denoted by $(a,b)_E$, which is defined as a $4$-dimensional $E$-algebra with basis $(1,i,j,k)$ endowed with the multiplication induced by the relations
$i^2=a$, $j^2=b$ and $k=ij=-ji$. 
Given an $E$-quaternion algebra $Q=(a,b)_E$, with $a,b\in\mg{E}$, we denote by $E(Q)$ the function field $E(X)(\sqrt{aX^2+b})$ over $E$, which is the function field of the projective conic over $E$ given by $aX^2+bY^2-Z^2=0$.
(More intrinsically, $E(Q)$ can be defined as the function field of the Severi-Brauer variety given by $Q$, see \cite[Section 5.4]{GS06}.)
A well-known theorem due to Witt (see \cite[Section 1.4]{GS06}) asserts that the isomorphism class of an $E$-quaternion algebra $Q$  
determines and is determined by the isomorphism class of the algebraic function field $E(Q)/E$. 
Hence it follows by \Cref{P:fufi-gen0-conic} that any regular function field of genus zero is isomorphic over $E$ to $E(Q)$ for some $E$-quaternion algebra $Q$.

Note that, as a consequence of Witt's Theorem, the $E$-quaternion algebra $Q$ is split (i.e.~isomorphic to the matrix algebra $\mathbb{M}_{2}(E)$) if and only if $E(Q)/E$ is a rational function field, and if this is not the case, then $Q$ is a division algebra, by \cite[Proposition 1.1.7]{GS06}.

Let now $v$ be again a valuation on $E$ with residue field $\kappa$ of characteristic different from $2$ and with value group $\Gamma$.

We refer to \cite[Chapter 1]{TW15} for basic facts from valuation theory over division rings. 
Recall that the axioms for defining a valuation can be formulated over any ring, but they imply that the ring has no zero-divisors. 
In this sense, given an $E$-algebra $D$, one can ask whether a valuation on $D$ (or an extension of $v$ to $D$) exists, but a positive answer will always require $D$ to have no zero-divisors, and thus a division algebra if $D$ is finite-dimensional.

\begin{lem}\label{L:residue-quatdivalg}
Let $a,b\in\mg{E}$ be such that $v(a)=v(b)=0$.
Let $Q=(a,b)_E$ and $\ovl{Q}=(\ovl{a},\ovl{b})_\kappa$.
If $\ovl{Q}$ is split, then $v$ does not extend to $Q$.
If $\ovl{Q}$ is a division algebra,
then $v$ has an unramified extension to $Q$.
\end{lem}
\begin{proof}
Assume first that $\ovl{Q}$ is split.
We may assume that $Q$ is a division algebra, since otherwise $v$ does not extend to a valuation on $Q$, by the definition.
As $\ovl{Q}$ is split, it follows by \cite[Proposition 1.1.7]{GS06} that $\ovl{b}=\ovl{z}^2-\ovl{a}\,\ovl{y}^2$ for some $\ovl{y},\ovl{z}\in {\kappa}$, and one can easily check that one can choose $\ovl{y}$ and $\ovl{z}$ to lie both in $\mg{\kappa}$.
Hence, there exist $y,z\in\mg{E}$ with $v(y)=v(z)=0$ such that $v(ay^2+b-z^2)>0$.
Consider $L=E(\sqrt{ay^2+b})$.
Note that $L$ is isomorphic over $E$ to maximal subfield of $Q$.
Since $v(ay^2+b)=v(z)=0$ and $\ovl{ay^2+b}=\ovl{z}^2\in\sq{\kappa}$, it follows by \Cref{L:FI2} that $v$ has two different extensions to $L$.
From this it follows by \cite[Section 1.2.2, Theorem 1.2]{TW15} that $v$ does not extend to $Q$.

Assume now that $\ovl{Q}$ is a division algebra.
This implies that $Q$ is a division algebra. 
(If $Q$ has zero divisors, then the conic $aX^2+bY^2=1$ has a solution over $K$, and this solution specializes to a solution of the conic $\ovl{a}X^2+\ovl{b}Y^2=1$ over $\kappa$, so also $\ovl{Q}$ has zero-divisors.)
The hypothesis implies that $v$ extends uniquely to a valuation on every maximal subfield of $Q$.
It follows by \cite[Section 1.2.2, Theorem 1.2]{TW15} that $v$ extends to a valuation $v_Q$ on $Q$.
Since the residue division ring of $v_Q$ is $\kappa$-isomorphic to $\ovl{Q}$, it follows by \cite[Section 1.2.1, Proposition 1.3]{TW15} that the value group of $v_Q$ is equal to $\Gamma$.
\end{proof}

The last statement can also be derived from \cite[Proposition 3.38]{TW15}.

\begin{prop}\label{P:infinite-rational}
Let $Q$ be an $E$-quaternion algebra and $F=E(Q)$.
Assume that $v$ has no unramified extension to $Q$.
Then there exist infinitely many extensions of $v$ to a valuation on $E(Q)$
whose residue fields are rational function fields over $\kappa$.
\end{prop}
\begin{proof}
Note that the hypothesis implies that the valuation $v$ on $E$ is nontrivial.

We first assume that $Q\simeq(a,b)_E$ for certain $a,b\in \mg{E}$ with $v(a)\notin 2\Gamma$.
We fix $x\in F$ such that $F=E(x)(\sqrt{ax^2+b})$.
For $c\in\mg{E}$ we denote by $w_{c}$ the Gauss extension of $v$ to $E(x)$ with respect to $cx$ and recall from \Cref{gaussextdef} that its value group is again $\Gamma$ and that $\kappa_{w_c}=\kappa(\ovl{cx})$, which is a rational function field over $\kappa$.
Note that for $c,c'\in\mg{E}$ with $v(c)\neq v(c')$ we have $w_c\neq w_{c'}$.
For any $c\in \mg E$ such that $2v(c)> v(a) -v(b)$, we obtain that
$w_c(ax^2)=v(a)-2v(c)<v(b)$ and thus $w_{c}(ax^2+b) = v(a)-2v(c) \notin 2\Gamma = 2\Gamma_{w_{c}}$, and by \Cref{L:FI2} this implies that $w_c$ the residue field of any extension of $w_c$ to $F$ is equal to $\kappa_{w_c}$.
Since the valuation $v$ is nontrivial, the set $\{\gamma\in\Gamma\mid 2\gamma>v(a)-v(b)\}$ is infinite, so we obtain in this way infinitely many extensions of $v$ to $F$ whose residue fields are rational function fields over $\kappa$.
Hence the statement is settled for this case.

We may henceforth assume that for any two elements $a,b\in\mg{E}$ with 
$Q\simeq(a,b)_E$ we have $v(a),v(b)\in 2\Gamma$.
In this case we may choose $a,b\in\mg{E}$ with $Q=(a,b)_E$ and $v(a)=v(b)=0$.
It now follows by \Cref{L:residue-quatdivalg} from the hypothesis that the $\kappa$-quaternion algebra $(\ovl{a},\ovl{b})_{\kappa}$ is split.
Hence there exist $\ovl y,\ovl z\in\kappa$ with $\ovl{b}=\ovl{y}^2-\ovl{a}\,\ovl{z}^2$, and one may choose $\ovl y$ and $\ovl z$ to lie in $\mg{\kappa}$.
In other words there exist $y,z\in\mg{E}$ with $v(y)=v(z)=0$ and
$v(ay^2+b-z^2)>0$.

As $Q=(a,b)_E$, we find an element $x\in F$ such that $F=E(x)(\sqrt{ax^2+b})$.
As $F/E$ is transcendental and $z\in E$, the element $x-z$ is transcendental over~$E$.
For any $c\in \mg{E}$, we denote by $w'_c$ the Gauss extension of $v$ to $E(x)$ with respect to $c\cdot(x-z)$.
As in the previous case we obtain for $c,c'\in \mg{E}$ with $v(c)\neq v(c')$ that $w'_c\neq w'_{c'}$.
Consider now $c\in \mg{E}$ with $v(c)<0$. 
Since $w'_c(x-z)=-v(c)>0$ it follows that $w'_c(x)=w'_c(z)=v(z)=0$ and $\ovl{x}=\ovl{z}\in\mg{\kappa}$.
Hence $w'_c(ax^2+b)\geq 0$ and $\ovl{ax^2+b}=\ovl{a}\cdot\ovl{z}^2+\ovl{b}=\ovl{y}^2\in\sq{\kappa}$.
Therefore $w'_c(ax^2+b)=0$, and we conclude by \Cref{L:FI2} that the residue field of any extension of $w'_c$ to a valuation on $F$ is equal to $\kappa_{w'_{c}}$, which is a rational function field over $\kappa$.
 Since the set $\{\gamma\in \Gamma\mid \gamma<0\}$ is infinite, we obtain in this way infinitely many extensions of $v$ to $F$ whose residue fields are rational function fields over $\kappa$.
\end{proof}

\begin{thm}\label{T:main}
Let $Q$ be an $E$-quaternion algebra and $F=E(Q)$.
Let  $v$ be a valuation on $E$ with residue field $\kappa$ of characteristic different from $2$.
Then the following are equivalent:
\begin{enumerate}[$(i)$]
\item The valuation $v$ extends to a valuation on $F$ whose residue field is transcendental and not ruled over $\kappa$.
\item The valuation $v$ extends uniquely to a valuation on $F$ whose residue field is transcendental and regular over $\kappa$.
\item The valuation $v$ has an unramified extension to a valuation on $Q$.
\end{enumerate}
If these conditions are satisfied, then the valuation in $(i)$ is unique and it coincides with the valuation characterised in $(ii)$.
\end{thm}

\begin{proof}
In view of \Cref{P:fufi-conic-gauss-residue}, we may fix $a,b\in\mg{E}$ such that $Q\simeq (a,b)_E$ and either $v(a)=v(b)=0$ or $v(a)\notin 2\Gamma$.
If $v(a)=v(b)=0$, then we further set $\ovl{Q}=(\ovl{a},\ovl{b})_\kappa$.
As $F=E(Q)$ there exists $x\in F$ such that $F=E(x)(\sqrt{ax^2+b})$.
We fix an extension $w_\star$ to $F$ of the Gauss extension of $v$ on 
$E(x^2)$ with respect to $\frac{ax^2}{b}$.
To structure the proof we introduce three more conditions:
{\it
\begin{enumerate}[$(i')$]
\item $\kappa_{w_\star}/\kappa$ is not ruled.
\item There exists no extension of $v$ to $F$ different from $w_\star$ whose residue field is transcendental and regular over $\kappa$.
\item $v(a)=v(b)=0$ and $\ovl{Q}$ is a division algebra.
\end{enumerate}}
We now show that all the conditions $(i)-(iii)$ and $(i')-(iii')$ are equivalent.
(Note that $(i')$ and $(ii')$ a priori depend on the choice of $w_\ast$, but of course this is not the case if they are equivalent with conditions that do not depend on this choice.)
Establishing these equivalences will further give that $w_\star$ is the only valuation on $F$ which can satisfy $(i)$ or $(ii)$, which confirms the last part of the statement.

\underline{$(i)\Leftrightarrow (i')$:} This follows from \Cref{non-ruled-val}.

\underline{$(ii)\Leftrightarrow (ii')$:} This is clear from \Cref{P:fufi-conic-gauss-residue}.

\underline{$(ii)\Rightarrow (iii)$:} This follows from \Cref{P:infinite-rational}.

\underline{$(iii)\Rightarrow(iii')$:}
Assume that $v$ has an unramified extension to $Q$.
Since $a$ is a square in $Q$, we obtain that $v(a)\in 2\Gamma$, whence $v(a)=v(b)=0$ in view of the choice of $a$ and $b$.
Since $v$ extends to $Q\simeq (a,b)_E$, it follows by \Cref{L:residue-quatdivalg} that $\ovl{Q}$ is a division algebra.

\underline{$(iii')\Rightarrow (i)$:}
Assume that $v(a)=v(b)=0$ and that $\ovl{Q}$ is a division algebra.
By \Cref{P:fufi-conic-gauss-residue}, we obtain that $\kappa_{w_\star}\simeq_\kappa \kappa(\ovl{Q})$. 
As $\ovl{Q}$ is a division algebra, it follows that $\kappa_{w_\star}/\kappa$ is not ruled.

\underline{$(iii')\Rightarrow (ii')$:}
Consider an arbitrary extension $w$ of $v$ to $F$ such that the residue field extension $\kappa_{w}/\kappa$ is transcendental and regular.
Since $\ovl{Q}$ is a division algebra, we have in particular that $\ovl{a},\ovl{b},-\ovl{ab}\notin\sq{\kappa}$. 
Since $\kappa_{w}/\kappa$ is regular, we deduce that $\ovl{a},\ovl{b},-\ovl{ab}\notin\sq{\kappa}_{w}$.
Since $\ovl{a}\notin\sq{\kappa}_{w}$, we have $w(x)\geq 0$, and  as $\ovl{b}\notin\sq{\kappa}_{w}$, we also have $w(x)\leq 0$.
Hence $w(x)=0$, and it follows that $w(ax^2+b)\geq 0$.
Since $-\ovl{ab}\notin\sq{\kappa}_{w}$, we further obtain that $w(ax^2+b)=0$.

Since $\kappa_{w}/\kappa$ is regular, the residue $\ovl{x}$ either lies in $\mg\kappa$ or it is transcendental over $\kappa$.
However,  assuming that $\ovl{x}\in\mg\kappa$,
we would obtain that $\ovl{a}\,\ovl{x}^2+\ovl{b}\in\sq{\kappa}$, 
which is impossible because $\ovl{Q}$ is a division algebra.
Hence $\ovl{x}$ is transcendental over~$\kappa$, and we conclude by \Cref{uniqueextensionlem} that $w=w_\star$.

\underline{$(i')\Rightarrow (iii)$:} Assume that $\kappa_{w_\star}/\kappa$ is not ruled.
By \Cref{non-ruled-val}, it follows that $\Gamma_{w_\star}=\Gamma$.
By \Cref{uniqueextensionlem} we obtain that $v(a)\in 2\Gamma$, whereby $v(a)=v(b)=0$ in view of our choice of $a$ and $b$.
It follows by \Cref{P:fufi-conic-gauss-residue} that $\kappa_{w_\star}\simeq_\kappa \kappa(\ovl{Q})$.
Since $\kappa_{w_\star}/\kappa$ is not ruled, it follows that $\ovl{Q}$ is a division algebra, and we conclude by \Cref{L:residue-quatdivalg} that $v$ has an unramified extension to $Q$.
\end{proof}

\subsubsection*{Acknowledgments}
We wish to express our gratitude to Nicolas Daans, Shira Gilat, David Grimm and Jean-Pierre Tignol for inspiring discussions and various valuable comments. 
We further wish to thank several referees for their careful reading and helpful input.

\bibliographystyle{amsalpha}

\begin{thebibliography}{1}
\bibitem{Deu}
M.~Deuring.
\emph{Lectures on the theory of algebraic functions of one variable}. Lecture
Notes in Math.~{\bf 314}. Springer-Verlag, Berlin--New York, 1973.

\bibitem{EKM}
R.~Elman, N.~Karpenko, A.~Merkurjev. \emph{The algebraic and geometric
  theory of quadratic forms}. AMS
  Coll Publ. {\bf 56}. Amer.~Math.~Soc., Providence, RI, 2008.

\bibitem{EP}
A.J. Engler, A.~Prestel, \emph{Valued fields}, Springer-Verlag, Berlin, 2005.


%

\bibitem{GS06}
P.~Gille, T.~Szamuely.
{\em Central Simple Algebras and Galois Cohomology}. Cambridge University Press, Cambridge, 2006.

\bibitem{KG}
S.~K.~Khanduja, U.~Garg. \emph{Residue fields of valued function fields of conics}. Proc.~Edinburgh Math.~Soc.
  \textbf{36}, (1993): 469--478.
%
\bibitem{Ohm}
J.~Ohm. \emph{The ruled residue theorem for simple transcendental extensions of
  valued fields}, Proc.~Amer.~Math.~Soc.~\textbf{89} (1983): 16--18.
  
  
\bibitem{TW15}
J.-P.~Tignol, A.R.~Wadsworth. \emph{Value functions on simple algebras, and associated graded rings.} Springer Monographs in Mathematics. Springer, Cham, 2015.
 
\end{thebibliography}

\end{document}